\documentclass[11pt  
]{amsart}
\usepackage{
amssymb 
}

\date{\today}

\newtheorem{theorem}{Theorem}
\newtheorem{proposition}{Proposition}
\newtheorem{lemma}{Lemma}
\newtheorem{definition}{Definition}

\theoremstyle{remark}
\newtheorem{remark}{Remark}

\DeclareMathOperator{\Vol}{Vol}

\DeclareMathOperator{\supp}{supp}
\DeclareMathOperator{\Ker}{Ker}

\DeclareMathOperator{\WF}{WF}
\DeclareMathOperator{\dist}{dist}

\newcommand{\eps}{\varepsilon}

\newcommand{\R}{{\bf R}}
\newcommand{\Id}{\mbox{Id}}
\renewcommand{\r}[1]{(\ref{#1})}
\newcommand{\PDO}{$\Psi$DO}
\newcommand{\be}[1]{\begin{equation}\label{#1}}
\newcommand{\ee}{\end{equation}}

\renewcommand{\d}{\mathrm{d}}

\renewcommand{\i}{\mathrm{i}}

\newcommand{\bo}{\partial \Omega}

\newcommand{\A}{\mathcal{A}}

\title[Multiwave tomography in a closed domain]{Multiwave tomography in a closed domain:\\ averaged sharp time reversal}

\author[Plamen Stefanov]{Plamen Stefanov}
\address{Department of Mathematics, Purdue University, West Lafayette, IN 47907}
\thanks{First author partly supported by a NSF  Grant DMS-1301646 }
\author[Yang Yang]{Yang Yang}
\address{Department of Mathematics, Purdue University, West Lafayette, IN 47907}

\usepackage{graphicx}
\graphicspath{%
    {converted_graphics/}
    {./}
}
\begin{document}
\begin{abstract}
We study the mathematical model of multiwave  tomography including thermo and photoacoustic tomography with a variable speed for a fixed time interval $[0,T]$. We assume that the waves reflect from the boundary of the  domain. We propose an averaged sharp time reversal algorithm. 
 In case of measurements on the whole boundary, we give an explicit solution in terms of a Neumann series expansion.  When the measurements are taken on a part of the boundary, we show that the same algorithm produces a parametrix. We present numerical reconstructions in both the full boundary and partial boundary data cases. 
\end{abstract}

\maketitle

\section{Introduction}  
The purpose of this work is to analyze the multiwave tomography mathematical model  when the acoustic waves reflect from the boundary and therefore the energy in the domain does not decrease. We model this with the energy preserving Neumann boundary conditions. This problem has been studied in the recent works  \cite{KunyanskiHC_2014,H_Kunyanski_14} motivated by the UCL  thermoacoustic group  experimental setup, see, e.g.,  \cite{CoxAB_07}. The papers \cite{KunyanskiHC_2014,H_Kunyanski_14} present numerical reconstructions and in \cite{H_Kunyanski_14}, the problem is analyzed  with the eigenfunction expansions method. That  approach requires a good control over the lower bound of the gaps between the Neuman eigenvalues and the Zaremba eigenvalues (or the Dirichlet ones in case of full boundary observations) which is not readily available, and cannot hold in certain geometries. It proposes a gradual asymptotic  time reversal as the observation time $T$ diverges to infinity,   which provides weak convergence under those conditions. On the other hand,  uniqueness and  stability for this problem are related to Unique Continuation and  Control Theory and  sufficient and necessary conditions for them follow from the Bardos-Lebeau-Rauch work \cite{BardosLR_control}. This was  noticed by Acosta and Montalto \cite{Acosta_M} who consider dissipative boundary conditions, including the case of Neumann ones we study. In the latter case, they propose a conjugate gradient numerical method; and if there is non-zero absorption, they show that a Neumann series approach similar to that in \cite{SU-thermo} can still be applied, even with partial data. 

Time reversal in its classical form  fails for this problem because the waves reflect from the boundary and there is no  good candidate for the Cauchy data at $t=T$, see section~\ref{sec_TR} below. In fact, doing time reversal at time $t=T$ with any choice of Cauchy data would produce a non-compact  error operator of norm at least one, as follows from our analysis; so it cannot be used even as a parametrix, see Figure~\ref{fig1}. The reason for the is the lack of absorption at the boundary either as absorbing boundary conditions or assuming a wave propagating to the whole space, as in the classical model; and this is the worst case for time reversal. A different problem arises when there is absorption in $\Omega$, see \cite{Andrew13}. 

In \cite{SU-thermo}, the first author and Uhlmann proposed a sharp time reversal method for the traditional thermo- and photo-acoustic model: when the acoustics waves do not interact with the boundary and propagate to the whole space, see also \cite{finchPR,Kruger03,Kruger99,KuchmentK_11,XuWang06}. 
The method 
 consists of choosing Cauchy data $(v,v_t)$ at $t=T$ that minimize the distance to the space of all Cauchy data $(f_1,f_2)$ with given trace on $\{T\}\times\bo$; and the latter is known from the data $\Lambda f$. This consists of choosing the Cauchy data $(\phi,0)$, where $\phi$ is the harmonic extension of $\Lambda f$ from $\{T\}\times\bo$ to $\{T\}\times\Omega$. Then we showed that the resulting error operator is a contraction, thus the problem can be solved by an exponentially and uniformly convergent Neumann series. Numerical simulations are presented in \cite{QSUZ_skull}. 

The main idea of this paper is to average the sharp time reversal done for times\footnote{we rename $T$ to $\tau$ below, and replace $T_0$ by $T$}  $T$ in an interval $[0,T_0]$ with $T_0>T(\Omega)/2$, where $T(\Omega)$ is greater than the length of the longest broken geodesic in $\bar\Omega$. 
The idea comes from the analysis of the error operators, see \r{I10} and Remark~\ref{rem3}. The latter propagates forward a wave with Neumann boundary conditions and sends back a wave reflecting according to the Dirichlet boundary conditions. It is well known that Neumann boundary conditions reflect the wave with no sign change, while the Dirichlet ones alter the sign, see section~\ref{sec_GO} for the microlocal equivalent of this phenomenon. While the error has norm one all the time, it has a sign depending on the time $T$. 
When we average, at $t=0$ we get  waves with the original signs and with the opposite ones, depending on the parity of the number of the reflections, see also Figure~\ref{fig:thermo}.  There is cancellation which makes the error operator a contraction, at least microlocally. The harmonic extension makes it an actual one. Those cancellations happen if and only if the stability condition implied by \cite{BardosLR_control} holds, and then we get an explicit reconstruction in the form of an exponentially convergent Neumann series, see Theorem~\ref{thm2.1}.  Also, instead of averaging multiple time reversals, we can average just one with an averaged boundary data $\phi(t)\Lambda f(t,x)$, see the first term in \r{h2} and also \r{T4}.  

The proposed algorithm can be applied to the partial data case as well. We time-reverse the Dirichlet data on the part $\Gamma$ of $\bo$ , where we have data; and imposed Neumann data on the rest. The Neumann series convergence then remains an open problem but we show that the method gives a parametrix away from a measure zero set when the stability condition is met. We present numerical reconstructions in both cases. 

\begin{figure}[h]  
  \centering
  \includegraphics[scale=0.15, keepaspectratio]{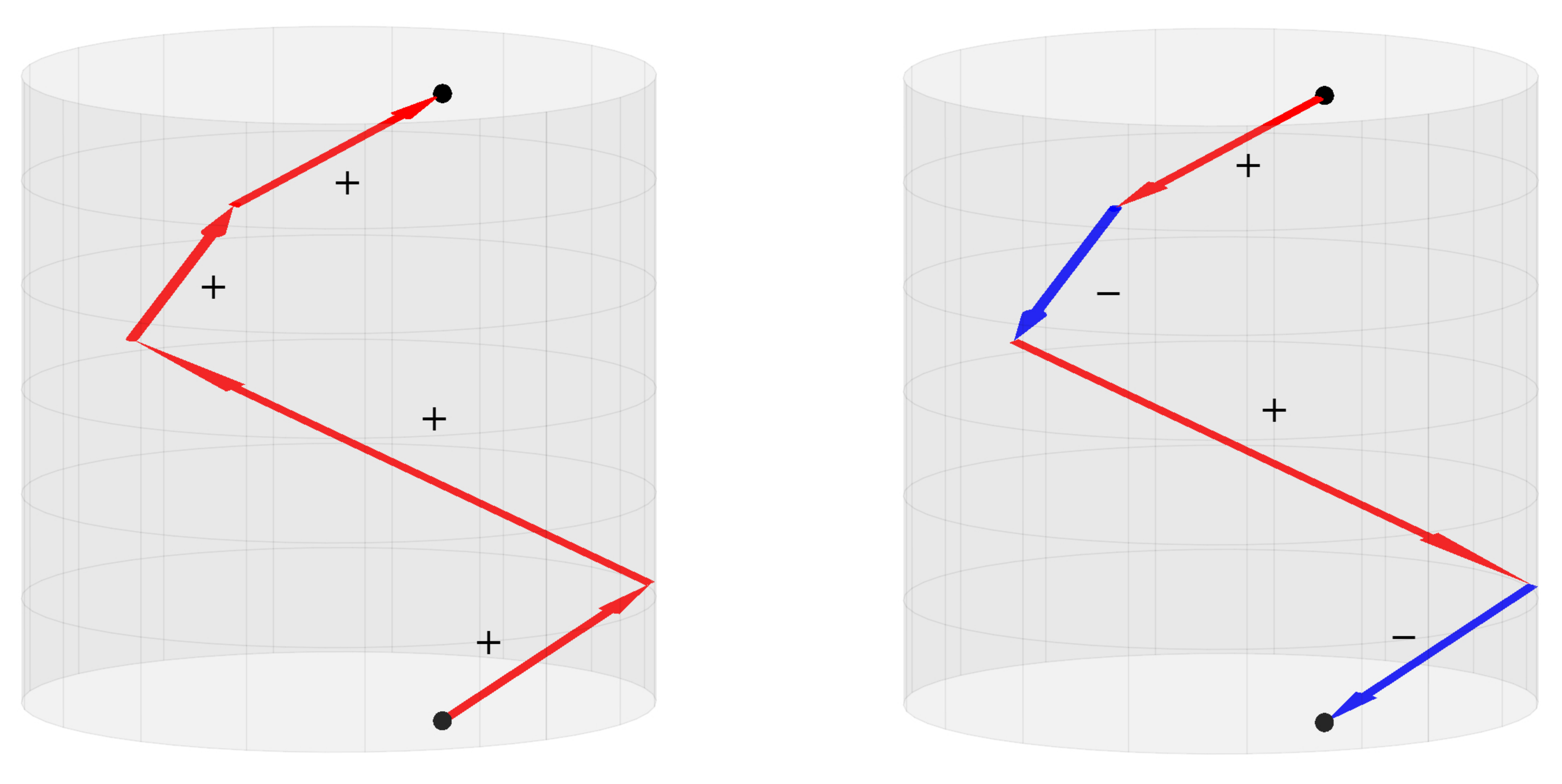}
  \caption{Propagation of singularities in $[0,T]\times\Omega$ for the positive speed only with Neumann boundary conditions (left) and time reversal with Dirichlet ones (right). In the latter case, the sign changes at each reflection.} \label{fig:thermo}
\end{figure}

For simplicity, we restrict ourselves to the case when the function we want to recover is supported in a fixed subdomain $\Omega_0\Subset\Omega$. Stability and uniqueness is unaffected by that, and already contained in \cite{BardosLR_control}. The micrlocal analysis justifying the time reversal however would be much more complicated without that assumption, and in applications, this condition is satisfied anyway. 

Our main results are the following. For full boundary measurements over time interval $[0,T]$ with a sharp $T$, we show in Theorem~\ref{thm2.1} that we can solve the problem by an exponentially convergent Neumann series. For partial data on $\Gamma\subset \bo$,  we show in Proposition~\ref{pr1} that if the stability (controllability) condition holds, our construction gives a parametrix away from the measure zero set of singularities which hit $\bo$ at the boundary of $\Gamma$ on $\bo$. Numerical reconstructions are presented on section~\ref{sec7} for both teh full and the partial data problems.

\textbf{Acknowledgments}. We would like to thank Carlos Montalto for his comments on a preliminary version of this paper and on providing us with a draft of  \cite{Acosta_M}. We would also like to thank Jie Chen and Xiangxiong Zhang for helpful discussions on numerical simulations.

\section{Preliminaries}

\subsection{The model}

Let $\Omega$ be a smooth bounded domain in $\R^n$. 
Let $g$ be a  Riemannian metric in $\bar\Omega$,   and let  $c>0$ be smooth. Let $P$ be the differential operator
\be{P}
P = -c^2\Delta_g,
\ee
where $\Delta_g$ is the Laplacian in the metric $g$. 
In applications, $g$ is Euclidean   but the speed $c$ is variable. For the methods we use, the metric $g$ poses no more difficulties than $P=-c^{-2}\Delta$. We could treat a more general second order symmetric operator involving a magnetic field and an electric one, as in \cite{SU-thermo} but for the simplicity of the exposition, we stay with $P$ as in \r{P}. The metric determining the geometry is $c^{-2}g$. 

Fix $T>0$. Let $u$ solve the problem
\begin{equation}   \label{1}
\left\{
\begin{array}{rcll}
(\partial_t^2 +P)u &=&0 &  \mbox{in $(0,T)\times \Omega$},\\
\partial_\nu u|_{(0,T)\times\bo}&=&0,\\
u|_{t=0} &=& f,\\ \quad \partial_t u|_{t=0}& =&0.
\end{array}
\right.               
\end{equation} 
Here $\partial_\nu = \nu^i \partial_{x^j}$, where $\nu$ is the unit, in the metric $g$, outer normal vector field on $\bo$.  The function  $f$ is the source which we eventually want to recover. The Neumann boundary conditions correspond to a ``hard reflecting'' boundary $\bo$. Let $\Gamma\subset\bo$ be a relatively open subset of $\bo$, where the measurements are made. The observation operator is then modeled by
\be{I1}
\Lambda f = u|_{[0,T]\times\Gamma}.
\ee
The methods we use allow us to treat  the case of Dirichlet boundary conditions in \r{1} and Neumann data  in \r{I1}.

\subsection{Function spaces}

The operator  $P$ is formally self-adjoint w.r.t.\ the measure $c^{-2}\d\Vol$, where $\d\Vol(x) = \sqrt{\det g}\, \d x$.  
Define the energy
\[
E(t,u) = \int_\Omega\left( |\nabla u|_g^2   +c^{-2}|u_t|^2 \right)\d\Vol,
\]
where   $|\nabla u|_g^2=g^{ij}(\partial_{x^i}u)(\partial_{x^j}u)$, and $\d\Vol(x) = (\det g)^{1/2}\d x$. This is just $(Pu,u)_{L^2}+\|u_t\|^2_{L^2}$ assuming that  $u$ satisfies boundary conditions allowing integration by parts without boundary terms.  Here and below, $L^2(\Omega) = L^2(\Omega; \; c^{-2}\d\Vol)$.

We define the Dirichlet space $H_{D}(\Omega)$ to be the completion of $C_0^\infty(\Omega)$ under the Dirichlet norm
\be{2.0H}
\|f\|_{H_{D}}^2= \int_\Omega |\nabla u|_g^2 \,\d\Vol.
\ee
Note that we actually integrate $|\nabla u|^2_{c^{-2}g}$ w.r.t.\ the volume measure of $c^{-2}\d\Vol$. 
By the trace theorem, the Dirichlet boundary condition $u=0$ on $\bo$ is preserved after the completion. 
It is easy to see that $H_{D}(\Omega)\subset H^1(\Omega)$, and that $H_{D}(\Omega)$ is topologically equivalent to $H_0^1(\Omega)$. Let $P_D$ be the Friedrichs extension of $P$ as self-adjoint unbounded operator $P_D$ on $L^2(\Omega)$ with domain $H_D\cap H^2$. For $f$ in the domain of $P_D$, we have  $\|f\|_{H_D(\Omega)}^2 = (P_Df,f)_{L^2}$. Note that the domain of the latter form is $H_D$, which a larger space than the domain of $P_D$.

To treat the Neumann boundary conditions, recall first that $P$, with Neumann boundary conditions, has a natural self-adjoint extension $P_N$ on $L^2$. First, one extends the energy form on $H^1(\Omega)$ (no boundary conditions) 
and then $P_N$ is the self-adjoint operator associated with that form, see \cite{Reed-Simon4}. The domain of $P_N$ is the closed subspace of $H^2(\Omega)$ consisting of functions $f$ with vanishing normal derivatives $\partial_\nu f$ on $\bo$.  
  In contrast to $P_D$, the operator  $P_N$ has a non-trivial null space consisting of the constant functions.  Such functions are stationary solutions of the wave equation and of no interest. Then we define $H_N(\Omega)$ as the   quotient space  $H^1(\Omega)/\Ker P_N$ equipped with the Dirichlet  norm. In other words, the functions in $H_N(\Omega)$ are defined up to a constant only. 
Note that on that space, $P_N$ is strictly positive. 
Both $P_D$ and $P_N$ are positive,  have compact resolvents, and hence point spectra only. They are both invertible. 

We can view $H_D(\Omega)$ as an equivalence class of functions constant on $\bo$; with two such functions equivalent if they differ by a constant (then they have the same norm). Then $H_D(\Omega)$ can be viewed as subspace of $H_N(\Omega)$.


The energy norm for the Cauchy data $(f_1,f_2)$, that we denote by $\|\cdot\|_{\mathcal{H}}$ is then defined by
\[
\|(f_1,f_2)\|^2_{\mathcal{H}} = \int_\Omega \left( |\nabla f_1|_g^2  +c^{-2}|f_2|^2 \right)\d\Vol.
\]
We define two  energy spaces 
\[
\mathcal{H}_D(\Omega) = H_D(\Omega)\oplus L^2(\Omega), \quad \mathcal{H}_N(\Omega) = H_N(\Omega)\oplus L^2(\Omega)
\] 
both equipped with the  energy norm defined above. We define the energy space $\mathcal{H}(\R^n)$ in $\R^n$ in a similar way; and we will use it only in our microlocal construction, with compactly supported functions. We denote pairs of functions below by boldface, like $\mathbf{f}=(f_1,f_2)$. Operators with range in vector valued functions will be denoted by boldface symbols, as well.

The wave equation then can be written down as the system
\be{s1}
\mathbf{u}_t= \mathbf{P}\mathbf{u}, \quad \mathbf{P} = \begin{pmatrix} 0&I\\P&0 \end{pmatrix},
\ee
where $\mathbf{u}=(u,u_t)$ belongs to the energy space $\mathcal{H}_D$ or $\mathcal{H}_N$. Choosing $P$ to be either   $P_D$ or $P_N$, we get  a skew-selfadjoint operator  $\mathbf{P}_D$, respectively $\mathbf{P}_N$ on $\mathcal{H}_D$, respectively $\mathcal{H}_N$, see \cite{Goldstein2003}. Those two operators generate unitary groups $\mathbf{U}_D(t) = \exp(t\mathbf{P}_D)$ and  $\mathbf{U}_N(t)= \exp(t\mathbf{P}_N)$, respectively.  
 
 Let $\mathcal{P}$ be the Poisson operator $\mathcal{P}:h\mapsto\phi$ defined as the solution of 
 \be{C1}
P \phi=0\quad \text{in $\Omega$}, \quad \phi|_{\bo} =h.
\ee
Of course, $Pu=0$ is equivalent to $\Delta_g u=0$. 
For $f\in H^1(\Omega) $, set 
\be{Pi}
\Pi f:= f-\mathcal{P}(f|_{\bo}).
\ee
Then $\Pi f$ vanishes on $\bo$. 
 By the trace theorem and standard energy estimates, $\Pi :H^1(\Omega)\to H_D(\Omega)$ is bounded. Also,  $\Pi = P_D^{-1}P$. One can think of $\Pi$ as an orthogonal  projection operator from $H_N(\Omega)$ to $H_D(\Omega)$ if we think of $H_D$ as an equivalence class as well, modulo constants as explained above. In any case, $\Pi$ is invariantly defined on $H_N(\Omega)$ as it is easy to see and we have the following. 

 \begin{lemma}\label{lemma_Pi} \
 
 (a) The operator $\Pi: H_N(\Omega)\to H_D(\Omega)$  has norm  $1$. 
 
(b) The operator $\mathbf{\Pi}: (f_1,f_2) \mapsto (\Pi f_1,f_2)$ from $\mathcal{H}_N(\Omega)$ to $\mathcal{H}_D(\Omega)$ has norm $1$. 
\end{lemma} 

\begin{proof} For $f\in H_N(\Omega)$, we have 
\[
f = \Pi f + \phi, \quad \phi =  \mathcal{P}(f|_{\bo}).
\]
This is an orthogonal decomposition w.r.t.\ the $H_D$ norm (which is only a seminorm on the second term). Therefore,
\[
\|f\|_{H_D}^2 = \|\Pi f\|_{H_D}^2+ \|\phi\|_{H_D}^2.
\]
This shows that the norm of $\Pi$ does not exceed $1$. Since we can take $f\not=0$ vanishing on $\bo$, the norm is actually $1$. 

The proof of (b) follows immediately from (a). 
\end{proof}

Similarly, let $\Omega_0\subset \Omega$ be a  subdomain with a smooth boundary. 
Identify $H_D(\Omega_0)$ with the subspace of $H_D(\Omega)$ of functions supported in $\bar\Omega_0$. Set 
$\Pi_0 f=h$ to be the solution of
\be{Ih}
Ph=Pf\quad\text{in $\Omega_0$}, \quad h|_{\bo_0}= 0.
\ee

 \begin{lemma}\label{lemma_Pi1} \
 $\Pi_0$ is the orthogonal projection from $H_D(\Omega)$ to $H_D(\Omega_0)$. 
 \end{lemma}
\begin{proof}
By standard energy estimates, $\Pi_0$ is bounded. Clearly, $\Pi_0^2=\Pi_0$. To compute the adjoint, choose $f_{1,2}\in C_0^\infty(\Omega)$ and write
\[
\begin{split}
(\Pi_0f_1,f_2)_{H_D(\Omega)} &= (\Pi_0 f_1,f_2)_{H_D(\Omega_0)} 
=\int_{\Omega_0} \langle\nabla \Pi_0 f_1,\nabla\bar f_2\rangle_g\,\d \Vol\\
& = (  \Pi_0 f_1, P f_2 )_{L^2(\Omega_0)}   =  (  \Pi_0 f_1, P\Pi_0  f_2 )_{L^2(\Omega_0)} =  (  \Pi_0 f_1, \Pi_0  f_2 )_{H_D(\Omega_0)}. 
\end{split}
\]
In the same way, we show that $(f_1,\Pi_0f_2)_{H_D(\Omega)}$ equals the same;  therefore, $\Pi_0$ is self-adjoint on a dense set, and therefore a self-adjoint (bounded) operator. Clearly $\Pi_0$ preserves $H_D(\Omega_0)$. This completes the proof. 
\end{proof}

\section{Uniqueness and stability. 
Relation to unique continuation and boundary control} \label{sec_US} 
We formulate below a sharp uniqueness result following from the uniqueness theorem of Tataru \cite{tataru95}. Next, we recall that a sharp stability condition (and some of the uniqueness results) have already been given in the work  \cite{BardosLR_control} by Bardos, Lebeau and Rauch. 

Assume in what follows that $f$ is supported in $\bar\Omega_0$, where $\Omega_0\subset\Omega$ is sone a priori fixed domain which could be the whole $\Omega$ in the uniqueness theorems but we will eventually require $\Omega_0\Subset\Omega$ (i.e., $\Omega_0$ is open and $\bar\Omega_0\subset\Omega$). 

\subsection{Uniqueness} The sharp uniqueness condition is of the same form as in \cite{SU-thermo} but the proof here is more straightforward. We want to allow a signal from any point to reach $[0,T]\times\Gamma$. That poses the following lower bound $T_0$ for the sharp uniqueness time:
\be{US1}
T_0 := \max_{x\in\bar\Omega_0}\dist(x,\Gamma). 
\ee
This bound is actually sharp, as the next theorem shows.
\begin{theorem}[Uniqueness] \label{thm_u} 
$\Lambda f=0$ for some $f\in H_D(\Omega_0)$ implies $f(x) =0$ for  $\dist(x,\Gamma)< T$. In particular, if $T\ge T_0$, then $f=0$. 
\end{theorem}

Clearly, if $T<T_0$, we cannot recover $f$ but we can still recover the reachable part of $f$. 

\begin{proof}
The proof follows directly from the unique continuation property of the wave equation, \cite{tataru95}. As in  \cite{SU-thermo_brain}, we have unique continuation from a neighborhood of any point on $\bo$ where we have Cauchy data. 
\end{proof}

\subsection{Stability} 
The stability condition is of microlocal nature, as it can be expected. The propagation of singularities theory, see section~\ref{sec_GO}, says that the singularities of $f$ starting from every point $(x,\xi)\in T^*\Omega\setminus 0$ split in two parts, propagating along the bicharacteristic  issued from $(x,\xi)$ and the other one along the bicharacteristic  issued from $(x,-\xi)$. The speed is one in the metric $c^{-2}g$, when the parameter is $t$.  
Those two singularities have equal energy, see the first identity in \r{o2f}.  The latter is due to the zero condition for $u_t$ at $t=0$. When each branch hits the boundary transversely, it reflects by the law of the geometric optics and the sign (and the magnitude) of the amplitude is preserved. The situation is more delicate when we have singularities with base points on $\bo$ or ones for which the corresponding rays hit $\bo$ tangentially. Then we can have a whole segment on $\bo$, called a gliding ray. The worst case is when they hit tangentially  concave points making an infinite contact with $\bo$. 
Those (non-smooth at $\bo$) curves are called generalized bicharacteristics and their projections to the base are called generalized geodesics. To avoid the difficulties mentioned above, we assume that $\bo$ is strictly convex w.r.t.\ the metric $c^{-2}g$ and that $\supp f\subset   \Omega$. Then all geodesics issued from $\supp f$ hit $\bo$ transversely (if non-trapping), and each subsequent contact is transversal, as well. The rays (the projections of the bicharacteristics on the base) then are piecewise smooth ``broken'' geodesics. We will formulate the analog of the Bardos-Lebeau-Rauch condition in this simpler situation. The only modification is to take into account that each singularity propagates in both directions with equal energy (what is important that neither of them is zero). Therefore, it is enough to detect only one of the two rays. 

\begin{definition} \label{def1}
Let $\bo$ be strictly convex with respect to $c^{-2}g$. 
Fix  $\Omega_0\Subset \Omega$, an open  $\Gamma\subset\Omega$ and $T>0$.

 (a)  We say that the stability condition is satisfied if every broken unit speed geodesic $\gamma(t)$  with $\gamma(0)\in \bar\Omega_0$ has at least one common  point  with $\Gamma $ for $|t|<T$, i.e., if    $\gamma(t)\in \Gamma$ for some $|t|<T$. 
 
 (b) We call the point $(x,\xi)\in T^*M\setminus 0$ a {visible singularity} if the unit speed geodesic $\gamma$ through $(x,\xi/|\xi|)$ has a common point with $\Gamma$ for $|t|<T$. We call the ones for which $\gamma$ never reaches $\bar\Gamma$ for $|t|\le T$ invisible ones. 
\end{definition} 

Common points of such  geodesics with $\bo$ are  the  points on $\bo$ where the geodesic reflects (transversely). Not that visible and invisible are not alternatives; the complement of their union  is the measure zero set of singularities corresponding to rays hitting $\bo$ for $|t|<T$ every time in  $\partial\Gamma$ only or hitting  $\bar\Gamma$ for the fist time  for $|t|=T$. 

Next theorem follows directly from \cite{BardosLR_control}, see Theorem~3.8 there.

\begin{theorem}\label{thm_st}
Let $\bo$ be strictly convex and fix  $\Omega_0\Subset \Omega$, an open $\Gamma\subset\Omega$ and $T>0$. Then if the stability condition is satisfied, 
\[
\|f\|_{H_D}\le C\|\Lambda f\|_{H^1((0,T)\times\Gamma)}.
\]
\end{theorem}

\section{Complete data. Review of the sharp time reversal} \label{sec_TR}
\subsection{Sharp time reversal}

Assume we have complete data, i.e., $\Gamma=\bo$ (but $T<\infty$). In what follows, we adopt the notation $u(t)=u(t, \cdot)$.  
In the time reversal step, to satisfy the compatibility conditions at $t=T$, we choose $v(T)$ to be the harmonic extension $\phi= \mathcal{P}(\Lambda f(T))$ of   $\Lambda f(T)$. Since $P=-c^2\Delta_g$, $\phi$ solves $P\phi=0$ as well  since $c^2$ cancels in the equation $-c^2\Delta_g\phi=0$, see \r{C1}. 
Solve 
\begin{equation}   \label{T1}
\left\{
\begin{array}{rcll}
(\partial_t^2 +P)v &=&0 &  \mbox{in $(0,T)\times \Omega$},\\
  v|_{(0,T)\times\bo}&=&h,\\
v|_{t=T} &=& \mathcal{P}h(T),\\ \quad \partial_t v|_{t=T}& =&0, 
\end{array}
\right.               
\end{equation} 
where, eventually, we will set $h=\Lambda f$, and set 
\be{I3}
Ah : = v(0). 
\ee
Then we think of $A\Lambda f$ as the time reversed data. In the multiwave tomography model in the whole space, $A\Lambda$ is often used as an approximation for $f$ at least when $T\gg1$. In our case, we cannot expect that but we still define the ``error'' operator $K$ by  
\[
A\Lambda = \Id-K. 
\]

To analyze $K$, let $w=u-v$ be the ``error''. Then $w$ solves
\begin{equation}   \label{I4}
\left\{
\begin{array}{rcll}
(\partial_t^2 +P)w &=&0 &  \mbox{in $(0,T)\times \Omega$},\\
  w|_{(0,T)\times\bo}&=&0,\\
w|_{t=T} &=& \Pi u(T)\\ \quad \partial_t w|_{t=T}& =& \partial_t u|_{t=T}. 
\end{array}
\right.               
\end{equation} 
Then
\be{I5}
Kf  = w(0). 
\ee
This yields the following for the operator $K : H_N \to   H_D$:
\be{I6}
K = \pi_1 \mathbf{U}_D(-T)\mathbf{\Pi} \mathbf{U}_N(T) \pi_1^* ,
\ee
where $\pi_1(f_1,f_2) := f_1$, $\pi_1^*f :=(f,0)$. Obviously,
\be{I7a}
\|K\|_{H_N\to H_D}\le 1. 
\ee
 We cannot expect $K$ to be a contraction anymore ($\|K\|<1$) for large $T$.  By constructing high-frequency solutions propagating along a single broken geodesic (in both directions), one can actually show that $\|K\|=1$ and finding $f$ from  $(\Id-K)f$ cannot be done in a stable way, at least.  This also follows from the analysis in Section~\ref{sec_GO}.  In Figure~\ref{fig1}, we present a numerical example illustrating what happens if we use that form of time reversal.

\subsection{A slightly different representation} 
 Set $\tilde v= v-\mathcal{P}(\Lambda f(T))$ with $v$ as in \r{T1}. Then $\tilde v$ solves
\begin{equation}   \label{T2}
\left\{
\begin{array}{rcll}
(\partial_t^2 +P)\tilde v &=&0 &  \mbox{in $(0,T)\times \Omega$},\\
 \tilde v|_{(0,T)\times\bo}&=&\Lambda f(t)   -\Lambda f(T)     ,\\
\tilde v(T) =\partial_t\tilde  v(T)& =&0. 
\end{array}
\right.               
\end{equation} 
Then 
\be{I7}
A\Lambda f =\tilde  v(0) + \mathcal{P}(\Lambda f(T)). 
\ee
Therefore, to compute $A\Lambda f$, we solve \r{T2} and then use \r{I7}.

When $f$ is a priori supported in $\bar\Omega_0$, we use $A_0:=\Pi_0A$ as time reversal, see Lemma~\ref{lemma_Pi1}. 
Then $A_0 \Lambda = \Pi_0-\Pi_0K$, and restricted to $H_D(\Omega_0)$, we have $A_0= \Id+K_0$, $K_0 := \Pi_0K$. Then in  \r{I6}, we apply $\Pi_0$ to the right to get $K_0$. 

\section{Averaged Time Reversal for complete data}
The main idea is to average the sharp time reversal above over a time interval. Let $T(\Omega)$ be the length of the longest geodesic in $\bar\Omega$. We assume that $\Omega$ is strictly convex with respect to $c^{-2}g$ and non-trapping, i.e., $T(\Omega)<\infty$. 

Fix $T>0$. Eventually, we will require  $T(\Omega)/2< T$. For  $\tau\le T$, let $A(\tau)$ be the time reversal operator $A$ constructed above with $T=\tau$. In \r{T2}, we can prescribe zero Cauchy data for $t>\tau$ and solve the problem on the interval $t\in [0,T]$ by extending the boundary condition $\Lambda f(t)   -\Lambda f(\tau)$ as zero for $t>\tau$ (which is a continuous extension across $t=\tau$). In other words, 
\be{I8}
A(\tau)\Lambda f =\tilde  v^\tau(0)+ \mathcal{P} (\Lambda f(\tau)), 
\ee
where $\tilde v^\tau $  solves (we drop the superscript $\tau$ below)
\begin{equation}   \label{T3}
\left\{
\begin{array}{rcll}
(\partial_t^2 +P)\tilde v &=&0 &  \mbox{in $(0,T)\times \Omega$},\\
 \tilde v|_{(0,T)\times\bo}&=& H(\tau -t) \left(\Lambda f (t)  -\Lambda f(\tau)  \right)   ,\\
\tilde v(T) =\partial_t\tilde  v(T)& =&0,
\end{array}
\right.               
\end{equation} 
where $H$ is the Heaviside function. 

Let $0\le \chi\in L^\infty([0,T])$ have integral one. 
Then we average $A(\tau)$ over $[0,T]$ with weight $\chi$. The result is an averaged time reversal operator
\be{A}
\mathcal{A} : = \int_{0}^{T} \chi(\tau) A(\tau)\,\d \tau. 
\ee
As explained in the Introduction, and will be proven in Section~\ref{sec_GO}, the averaged time reversal  restores all singularities with positive but not necessarily equal amplitudes, see also Figure~\ref{fig2}, which is the improvement we seek. 

To compute $\A$, we average both sides of \r{T3}. In other words, we do the time reversal in \r{T3} with boundary condition
\be{TR}
\begin{split}
 h(t):=    & \int_{0}^{T}  \chi(\tau)  H(\tau-t) \left(\Lambda f(t)   -\Lambda f(\tau) \right)\,\d \tau\\
 & =     \int_t^{T}    \chi(\tau)  \left(\Lambda f(t)   -\Lambda f(\tau) \right)\,\d \tau .
 \end{split}
\ee
Then we solve
\begin{equation}   \label{T4}
\left\{
\begin{array}{rcll}
(\partial_t^2 +P) v &=&0 &  \mbox{in $(0,T)\times \Omega$},\\
  v|_{(0,T)\times\bo}&=& h(t)  ,\\
 v(T) =\partial_t  v(T)& =&0,
\end{array}
\right.               
\end{equation} 
and set 
\be{AA}
\mathcal{A}\Lambda f = v(0) + \mathcal{P} \int_0^T\chi(\tau)\Lambda f(\tau)\,\d\tau .
\ee
Next, we project the result onto $H_D(\Omega_0)$ by taking $\Pi_0\mathcal{A}\Lambda f$ to be our time reversed version.  
The projection of the last term above vanishes (because it is harmonic), and we get
\[
\mathcal{A}_0\Lambda f = \Pi_0 v(0).
\]

 Let us also note that $h(t)$ can be expressed as
\be{h2}
 h(t) = \int_t^T\chi(\tau)\,\d\tau \cdot \Lambda f(t) -  \int_t^T \chi(\tau)\Lambda f(\tau)\,\d\tau. 
 \ee 

The next theorem gives an explicit inversion of $\Lambda$ on functions a priori  supported in $\Omega_0$. 
\begin{theorem}  \label{thm2.1} Let $(\Omega, c^{-2}g)$ be non-trapping, strictly convex, and let  $T(\Omega)/2< T$. Let $\Omega_0\Subset\Omega$. 
Then $\mathcal{A}_0\Lambda=\Id-\mathcal{K}_0$ on $H_D(\Omega_0)$, where $\mathcal{K}_0$ is compact in $H_{D}(\Omega_0)$, and   $\|\mathcal{K}_0\|_{\mathcal{L}(H_{D}(\Omega_0))}<1$. 
In particular, $\Id-\mathcal{K}_0$ is invertible on $H_{D}(\Omega_0)$, and the inverse  problem has an explicit solution of the form
\be{2.2}
f = \sum_{m=0}^\infty \mathcal{K}_0^m \mathcal{A}_0h, \quad h:= \Lambda f.
\ee
\end{theorem}

\begin{proof}
We divide the proof in several steps.  

(i) We notice first that 
\[
\|\mathcal{K}_0\|_{\mathcal{L}(H_{D}(\Omega_0))} \le1.
\]
This follows immediately from the following, see \r{I6}
\be{I10}
\mathcal{K}_0 =  \Pi_0\int_{0}^{T} \chi(\tau) \pi_1\mathbf{U}_D(-\tau)\mathbf{\Pi U}_N(\tau)\pi_1^*\,\d \tau
\ee
and Lemma~\ref{lemma_Pi}. Indeed, for $f\in H_D(\Omega_0)$, 
\be{I12}
\begin{split}
\|\mathcal{K}_0f\|_{H_D(\Omega_0)}&\le 
 \int_{0}^{T}\chi(\tau) \|\pi_1\mathbf{U}_D(-\tau)\mathbf{\Pi} \mathbf{U}_N(\tau)\pi_1^*f\|_{H_D(\Omega_0)} \,\d \tau \\
&\le \int_{0}^{T}  \chi(\tau) \|\mathbf{\Pi}\mathbf{ U}_N(\tau)\pi_1^*f\|_{H_D(\Omega)}\,\d \tau\\
&\le  \int_{0}^{T} \chi(\tau)  \|\mathbf{U}_N(\tau)\pi_1^*f\|_{H_N(\Omega)}\,\d \tau = \|f\|_{H_D(\Omega_0)}. 
\end{split}
\ee

(ii) By unique continuation, 
\be{I11}
\|\mathcal{K}_0f\|_{H_D(\Omega_0)}<\|f\|_{H_D(\Omega_0)}, \quad f\not=0. 
\ee
Indeed, if we assume equality above, then all inequalities in \r{I12} are equalities. Then
\[
\mathbf{\Pi U}_N(\tau)\pi_1^*f =  \mathbf{U}_N(\tau)\pi_1^*f, \quad  0\le \tau \le T. 
\]
By the the definition of $\mathbf{\Pi}$ in  Lemma~\ref{lemma_Pi}, the first component of the right hand side must vanish on $\bo$, i.e., 
$\pi_1 \mathbf{ U}_N( \tau)\pi_1^*f=0$ on $(0,T)\times\bo$. 
By the uniqueness theorem, $f=0$ since $T>T(\Omega)/2$ implies that the condition on $T$ in Theorem~\ref{thm_u} is satisfied. 

(iii)  
The essential spectrum of  $\mathcal{K}_0^*\mathcal{K}_0$ is included in $[0,1-\epsilon]$ with some $\epsilon>0$. 
We prove this below in Lemma~\ref{lemma_R}. 

(iv) The operator $\mathcal{K}_0$ is a contraction, i.e., 
\[
\|\mathcal{K}_0\|_{\mathcal{L}(H_{D}(\Omega_0))}<1. 
\]
It is enough to prove that the self-adjoint operator $\mathcal{K}_0^*\mathcal{K}_0$ is a contraction. 
By (iii) above,  the spectrum of $\mathcal{K}_0^*\mathcal{K}_0$ near $1$ consists of eigenvalues only, see \cite[VII.3]{Reed-Simon1}. One the other hand, $1$ cannot be an eigenvalue because then for the corresponding eigenfunction $\phi$ we would have $\mathcal{K}_0^*\mathcal{K}_0 \phi=\phi$, therefore, $\|\mathcal{K}_0 \phi\|^2=\|\phi\|^2$, which contradicts \r{I11}. 
\end{proof}

A numerical reconstruction with a variable speed based on the theorem is presented in Figure~\ref{fig3}.

\section{Geometric Optics and proof of the main lemma} \label{sec_GO}
We recall here some well-known facts about the reflection of singularities of solutions of the wave equation for transversal rays; both for the Dirichlet and the Neumann boundary conditions. In what follows, the notation  ${A}\cong{B}$ for two operators in ${H}_D$ indicates that they differ by a compact one. Similarly, ${Af}\cong{Bf}$ means that ${Af}\cong{(B+K)f}$ with ${K}$ compact. All \PDO s will be applied to functions supported in $\bar\Omega_0$ and will be assumed to have a Schwartz kernels supported in $\Omega\times\Omega$. Also, $|\cdot|_g$ is the norm of a vector or a covector, depending on the context, in the metric $g$; while $|\cdot|$ is the norm in the metric $c^{-2}g$.

A parametrix for the solution of the wave equation $(\partial_t^2-c^2\Delta_g)u=0$ with Cauchy data $(u,u_t)=(f_1,f_2)$ at $t=0$ in the whole space is constructed as
\be{o1}
u(t,x) =  (2\pi)^{-n} \sum_{\sigma=\pm}\int e^{\i\phi_\sigma(t,x,\xi)}\left( a_{1,\sigma}(x,\xi,t) \hat f_1(\xi)+  |\xi|_g^{-1}a_{2,\sigma}(x,\xi,t) \hat f_2(\xi)\right) \d \xi,
\ee
modulo  terms involving smoothing operators of $f_1$ and $f_2$. The reasons for the two terms is that the principal symbol $-\tau^2+c^2|\xi|_g^2$ of the wave operator has two smooth components  (away from the origin) of its characteristic variety: $\Sigma_\pm := \{\tau \pm c|\xi|_g=0\}$. Based on \r{o1}, we can write $u=u_++u_-$ (modulo smoothing terms), where $u_\pm$ solves the ``half wave equation'' $(-\i\partial_t\pm \sqrt{P})u_\pm=0$. 
The initial conditions are 
\[
(u_\pm,\partial_t u_\pm) |_{t=0}= \mathbf{\Pi}_\pm (f_1,f_2),
\]
where
\be{Pi'}
\mathbf{\Pi}_+ = 
\frac12 \begin{pmatrix}  1   &-\i P^{-1/2} \\\i P^{1/2}&1 \end{pmatrix},\quad 
\mathbf{\Pi}_- = \frac12 \begin{pmatrix}  1   &\i P^{-1/2} \\ -\i P^{1/2}&1 \end{pmatrix}.
\ee
In those equations, $P^{-1/2}$ can be considered as parametrix but the equations  are actually exact in $\R^n$, see the appendix in \cite{S-U-InsideOut10}. Note that $\mathbf{\Pi}_\pm$ are orthogonal projections in $\mathcal{H}(\R^n)$ and their sum is identity. The orthogonality is preserved under the dynamics in $\R^n$. 

The phase functions above solve the eikonal equations $\partial_t\phi_+ \pm c^2| \partial_x\phi_+ |_g=0$, $\phi_\pm|_{t=0}= x\cdot\xi$. If $P=-\Delta$, then $\phi_\pm = \mp |\xi|t+x\cdot\xi$; and the zero bicharacteristics are given by $(t,x)=(t_0+s,x_0\pm s\xi/|\xi|)$, with $\tau$ and $\xi$ constant and on $\Sigma\pm$.  
   That equation can be solved in general  for small $t$  only  if $(x,\xi)$ are restricted to a compact set. The amplitudes $a_{j,\sigma}$ are classical, of order $0$ solve the corresponding transport equations below and their leading terms  satisfy the initial conditions
\be{o2f}
a_{1,+}^{(0)}=a_{1,-}^{(0)}=\frac12,\quad  a_{2,+}^{(0)}=-a_{2,-}^{(0)}= \frac{\i}{2c(x)}\quad\text{for $t=0$}. 
\ee
The  transport equations for the principle terms have the form 
\be{tr}
\left( (\partial_t \phi_\pm) \partial_t - c^{2}g^{ij} \partial_{x^i}\phi_\pm \partial_{x^j}+C_\pm \right)a_{j,\pm}^{(0)}=0, 
\ee
with $C_\pm$ a smooth multiplication term. 
This is an ODE along the vector field $(-\partial_t\phi_\pm, c^2g^{-1}\partial _x\phi_\pm )$ (the second term is just the covector $\partial_x \phi_\pm$ identified  with a vector by the metric $c^{-2}g$) , and the integral curves of it coincide with the geodesic  curves $(t,\gamma_{z,\pm \xi}(t))$, with the metric identification of the tangent and the cotangent bundle. Given an initial condition at $t=0$, it has a unique solution along the integral curves as long as $\phi_\pm $ is well defined. Here and below, we denote by $\gamma_{z, \xi}(t)$ the   geodesic through $(x,\xi)$.

Assume that the wave front of $(f_1,f_2)$ is contained in a small conic neighborhood of some $(x_0,\xi^0)$. Let the constriction above be valid in some neighborhood of the segment of $(t,\gamma_{x_0,\xi^0}(t))$  until it hits $\R_+\times\bo$ (and  a bit beyond). We will work with $u_+$ only that we call just $u$, and we drop the subscript $+$ for the phase function, etc., below. 
 Let $u^R(t)$ be the reflected $u$ constructed as follows. 
Define the exit time $\tau$ by the condition
\be{tau}
\tau_\pm (x,\xi) = \left\{ \pm t\ge0;\;\gamma_{x, \xi}(t)\in\bo \right\}.
\ee
The function $\tau$ is positively homogeneous in $\xi$ or order $-1$. 
Fix boundary normal coordinates $(x',x^n)$ on $\bo$ near the reflection point so that $x^n=0$ defines $\bo$ locally and $x^n<0$ in $\Omega$, and the metric $c^{-2} g$ takes the form $(\d x^n)^2 + g'_{\alpha\beta}\d x^\alpha \d x^\beta$, $0\le\alpha,\beta\le n-1$. 
Restrict $u$ to $\R_+\times\bo$. Then $u$ would look like the $\sigma=+$ term in \r{o1} with $\phi=\phi|_{x^n=0}$, $a_{j}=  a_{j}|_{x^n=0}$ (recall that we dropped the $+$ subscript).  The map $F: (u,\partial_tu)|_{t=0}\mapsto u|_{\R_+\times\bo}$ (which is the operator $\Lambda$ in the commonly used model in the whole space, microlocally restricted) is an FIO of order $(0,-1)$ with a canonical relation associated to the diffeomorphism  $C_+$, where \cite{SU-thermo}:
\be{CR}
C_\pm :(x,\xi) \longmapsto\left(\pm \tau_\pm(x,\xi/|\xi|), \gamma_{x,\xi}( \tau_+(x,\xi)),\mp|\xi|,\dot\gamma'_{x,\xi}(\tau_\pm(x,\xi))\right),
 \ee
 where the prime stands for a projection on $T^*\bo$, and we identify vectors and covectors by the metric $c^{-2}g$. The map $C_-$ corresponds to $u_-$. 
The range of $C_\pm$ is in a compact subset of the hyperbolic regions $|\xi'|<\mp \tau$  of $T^*(\R\times\bo)$.  
 
  Its parametrix is the backprojection:  $ u|_{\R_+\times\bo}\mapsto (u,\partial_t u)|_{t=0}$ constructed as the restriction to $t=0$  of the incoming solution  of the boundary value problem (the one with smooth Cauchy data for $t\gg1$)  and boundary data  $ u|_{\R_+\times\bo}$, see also \cite{SU-SAR}. 
  
  We seek a parametrix for the reflected solution $u^R$ in the form
\be{o2}
u^R(t,x) =  (2\pi)^{-n} \int e^{\i\phi^R(t,x,\xi)}\left( a_{1}^R(x,\xi,t) \hat f_1(\xi)+  |\xi|_g^{-1}a_{2}^R(x,\xi,t) \hat f_2(\xi)\right) \d \xi.
\ee
In other words, $u^R = RF (f_1,f_2)$, where $R$ is the reflection operator, defined correctly because $F$ is microlocally invertible.  The phase function solves the eikonal equation 
\be{1.24a}
\partial_t\phi^R +c(x)|\nabla_x\phi^R|_g=0,\quad \phi^R|_{x^n=0}=\phi.
\ee
One such solution is $\phi$ itself (denoted above by $\phi_+$) and $\phi^R$ is the other one. They can be distinguished by the sign of their normal derivatives $\partial_\nu \phi$, $\partial_\nu \phi^R$ on $\R_+\times \bo$  which is positive for $\phi$ and negative for $\phi^R$. That derivative is as in \r{symb} without the factor $i$.  The amplitudes $a_1^R$ and $a_2^R$ solve the transport equations with initial data on $\R_+\times \bo$ equal to $a_1$ and $a_2$, respectively. Not that those transport equations are ODEs along the reflected geodesic. 

Since $\partial_\nu u$ and $\partial_\nu u^R$ have opposite signs of their principal terms, 
$u_N:= u+u^R$  satisfies the Neumann boundary conditions up to lower order terms. One can construct the whole reflected amplitudes this way but for our purposes, we just need the ``error'' term to be a compact operator. 
On the other hand, $u_D := u-u^R$ satisfies the Dirichlet boundary condition up to lower order terms. 
 
In particular, we recover the well know fact that Neumann boundary condition reflects the ``wave'' without a sign change, while the Dirichlet boundary condition alters the sign.

We can remove the condition now that the geometric optics construction \r{o1} is valid all the way to the boundary. The map from the Cauchy data $(f_1,f_2)$ to the solution $\mathbf{u}(t)$ at any given $t$ is an invertible FIO. Fix $t_0$ not exceeding the time it takes for the geodesic $(t,\gamma_{x_0,\xi^0})$ to hit the boundary but close enough to it. Then we repeat the arguments above with $u$ as in \r{o2} but with Cauchy data $(u,\partial_t u)|_{t=t_0}$ at $t=t_0$.

That phenomenon can be understood by studying the corresponding outgoing Dirichlet-to-Neumann (DN) map $N_\text{out}: u|_{\R\times\bo}\mapsto \partial_\nu u|_{\R\times\bo}$ for $u$ smooth for $t\ll0$) and the corresponding incoming one $N_\text{in}$, defined in the same way but requiring $u$ to be smooth for $t\gg0$. As follows from \r{o2} (and it is well known in scattering theory), they are both \PDO s on the hyperbolic conic set $c|\xi'|_g<|\tau|$ where the range of $C_\pm$ belongs, see \r{CR}, with opposite same principal symbols. The representation \r{o2}, see \cite{SU-thermo_brain} for details, implies that those principal symbols are 
\be{symb}
\pm \i\sqrt{c^{-2}\tau^2-|\xi'|^2_g},
\ee
where the positive sign is for the incoming one. In the hyperbolic conic set $c|\xi'|_g<|\tau|$, those symbols are elliptic, therefore $N_\text{out}^{-1}N_\text{in}\cong-\Id$ and $N_\text{in}^{-1} N_\text{out}\cong -\Id$ modulo lower order \PDO s, with the inverse meaning a parametrix.  Then in the construction above, given $u$ on $\bo$, we seek $u_R$ on the boundary  as the solution of $N_\text{in} u+ N_\text{out}u_R=0$ which implies $u_R\cong u$ on the boundary; hence  the Dirichlet data of the reflected solution is twice that of the incoming one, modulo lower order terms. 

We deifine another relevant map. Given boundary data $h$ microlocalized near some $(t,x,\tau,\xi')$ in the hyperbolic domain $|\xi'|<-\tau$ (related to the positive sign in \r{CR}), let $u$ be the outgoing solution (smooth for $t\ll1$) with that boundary data, extended until the corresponding geodesic hist $\bo$ again, and slightly beyond. Let $Gh= u|_{\R\times\bo}$ be the trace on the boundary there. Then $G$ is an elliptic FIO of order zero with a canonical relation given by the graph of
\be{CR1}
C_b :(t,x,\tau,\xi) \longmapsto\left(t+ \tau_+(x,\xi/|\xi|), \gamma_{x,\xi}( \tau_+(x,\xi)),\mp|\xi|,\dot\gamma'_{x,\xi}(\tau_\pm(x,\xi))\right).
 \ee

We are ready to analyze the reflection of singularities now. 
We can represent the solution of the forward Neumann problems as 
\be{chain1}
\mathbf{f}= (f_1,f_2)\quad \longmapsto  2F\mathbf{f}   \quad  \longmapsto\quad   2GF\mathbf{f}\quad  \longmapsto\quad 2G^2F\mathbf{f} \quad 
\longmapsto \dots,
\ee
where we start with the Cauchy data at $t=0$, the second term  is the Dirichlet data at the first reflection near $\tau_1:=\tau(x,\xi)$;  then the second reflection near $\tau_2(x,\xi)$, etc. 

To understand the backprojection with given Dirichlet data, note first that 
the backprojection of the Dirichlet data of $u_N\cong 2u_1$ on $\bo$ near the first reflection  to $t=0$ is just $\cong F^{-1}2F\mathbf{f}\cong  2\mathbf{f}$, i.e., we get  
\be{chain1a} 
u_N|_{\R_+\times\bo, \text{near 1st reflection}}\quad  \longmapsto\quad 2\mathbf{f}.
\ee
   Let us backproject the Dirichlet data $u_N|_{\R_+\times\bo}$ (for $t$ near $\tau_2(x,\xi/|\xi)$) in  \r{chain1} \textit{under the assumption that at the first reflection near $\tau_1$ the Dirichlet condition is zero}. We get
   
\be{chain2}
u_N|_{\R_+\times\bo, \text{near 2nd reflection}}  \quad \longmapsto\quad  0|_{\R_+\times\bo, \text{near 1st reflection}} \quad    \longmapsto\quad -2\mathbf{f}.
\ee
Now, the backprojection of both singularities $2u_2|_{\R_+\times\bo}$ and $2u_1|_{\R_+\times\bo}$ is a sum of the ones above, and we get $0$, modulo lower order terms.

We can continue this construction to get the following. Backprojecting even number of Dirichlet data of a singularity at consecutive reflections returns $0$ (therefore, an error operator $K=\Id$); and backprojecting an odd number returns $2\mathbf{f}$ (therefore, an error operator $K=-\Id$). This is consistent with the analysis above. 

In the proof below, we would need to backproject Dirichlet data multiplied by a smooth function $\phi(t)$. Given boundary data $h$ microlocally supported near the first reflection point of $\gamma_{x_0,\xi^0}$,  $F^{-1}$ is the back-projection \r{chain1}. 
Then by Egorov's theorem, $F^{-1}\phi h \cong (\phi \circ \tau_+)(x,D/|D|)F^{-1} h$. Then \r{chain1a} takes the form 
 \be{chain1b} 
\phi(t)u_N|_{\R_+\times\bo, \text{near 1st reflection}}\quad  \longmapsto\quad 2 (\phi \circ \tau_+)(x,D/|D|) \mathbf{f}.
\ee
To generalize \r{chain2} in this setting, note that the sequence of maps there is to apply $G^{-1}$, then $-\Id$ at the time of the first reflection, then $F^{-1}$. All those are FIOs associated to canonical diffeomorphisms, so we get 
\be{chain2a}
\begin{split}
\phi (t)u_N|_{\R_+\times\bo, \text{near 2nd reflection}}  \quad &\longmapsto\quad  0|_{\R_+\times\bo, \text{near 1st reflection}}\\\quad    &\longmapsto\quad -2(\phi \circ \tau_2)(x,D/|D|)F^{-1}\mathbf{f},
\end{split}
\ee
where $\tau_2(x,\xi)$ is the time of the second reflection of $\gamma_{x,\xi}(t)$. 
 
We therefore proved the following. 
 \begin{lemma}\label{lemma_main}
Let $v$ solve \r{T4} with $h$  given by
\be{h1}
h(t) = \phi(t)\Lambda(t)f,
\ee
(compare with  \r{h2}), with some $\phi\in C_0^\infty([0,T])$.  Then the map
\[
f\mapsto v(0)
\]
is a classical \PDO\ in $\Omega$ of order $0$ with principal term $p(x,\xi/|\xi|)$, with 
\[
p = \dots +\phi\circ \tau_{-3}-\phi\circ\tau_{-2}+  
(\phi\circ\tau_{-1}+\phi\circ\tau_1) -\phi\circ\tau_2+\phi\circ\tau_3+ \dots,
\]
where $\dots<\tau_{-2}<\tau_{-1} <0<\tau_1<\tau_2<\dots$ 
are the reflection times (except for $0$) of the unit speed geodesic issued from $(x,\xi)$, and $\phi$ is extended as an even function to $t<0$. 
\end{lemma}

Let $0\le \phi$ be decreasing, as in our main result, with $\phi(0)=1$. Then $p\ge0$. On the other hand it is straightforward to see that $p\le2$; and $0<p<2$ if $\phi$ is strictly increasing and $\phi\circ\tau_k>0$ for at least one $k$. Therefore, the error $1-p$ is in $(0,1)$ for every fixed $(x,\xi)$. 
 
In what follows, we restrict $(x,\xi)$ to the unit cosphere bunlde $S^*\Omega_0$.  
 
We think of $l_k:= \phi\circ\tau_{k+1}- \phi\circ\tau_k$, $k\ge1$ as the weighed time between the $k$-th and the $(k+1)$-th reflection with weight $-\phi'\ge0$. Since $\phi(0)=1$, the weighted time between $\tau_{-1}$ and $\tau_1$ should be $(1-\phi\circ\tau_1)  +(1- \phi\circ\tau_{-1})$. This motivates the following definition:
\be{tau'} 
l_k=\begin{cases}
\phi\circ\tau_{k}- \phi\circ\tau_{k+1},& k=1,2,\dots,\\
2-\phi\circ\tau_{1}- \phi\circ\tau_{-1},  & k=0,\\
\phi\circ\tau_{-k}- \phi\circ\tau_{-k-1},  & k=-1,-2,\dots.
\end{cases}
\ee
Then 
 \[
 \begin{split}
\mu :&= \dots+l_{-2}-l_{-1}+l_0- l_1+l_2+l_3-\dots\\
& = \dots+( -\phi\circ\tau_{-3}+\phi\circ\tau_{-2}  )- ( -\phi\circ\tau_{-2}+\phi\circ\tau_{-1}  )\\
&\qquad\quad + ( 2-\phi\circ\tau_{-1}-\phi\circ\tau_{1}  )  -( \phi\circ\tau_{1}-\phi\circ\tau_{2}  )+\dots \\
&= \dots- 2\phi\circ\tau_{-3}+2\phi\circ\tau_{-2}- 2(\phi\circ\tau_{-1} +\phi\circ\tau_{-1}) +2\phi\circ\tau_{2}-  2\phi\circ\tau_{3}+\dots\\
  & = 2(1-p).
\end{split}
 \]
 On the other hand, 
 \[
  \dots+l_{-2}+l_{-1}+l_0+ l_1+l_2+l_3\dots=2.
 \]
Thus we get the following.
\begin{lemma}\label{lemma_p}
Under the assumptions on $\phi$, on $S^*\bar\Omega_0$, 
\be{kappa}
p =1-\kappa, \qquad \text{where}\quad 
\kappa = \sum_{k=-\infty}^\infty (-1)^kl_k\Big/ \sum_{k=-\infty}^\infty l_k.
\ee
If $\phi\circ\tau_k(x,\xi)\not=0$ for at least one $k$ (i.e., for $k=1$ or $k=-1$), then $|\kappa(x,\xi)|<1$. 
\end{lemma}
 
 To understand better this lemma, consider a few special cases.
 \begin{remark}\label{rem_T1}
 Let $\phi$ be the characteristic function of $[-T,T]$, which is not smooth, of course but we can always cut it off smoothly near the endpoints which does not change our conclusions below if neither $\tau_k$ can be $\pm T$. This corresponds to non-averaged time reversal. Then $\phi\circ\tau_k=1$ for every $\tau_k$ in the interval $[-T,T]$. Thus $l_k$ are either zero if the whole interval is in $[-T,T]$ or $1$ otherwise. 
Then $p$ take values $0$, $1$ of $2$ depending on the number of reflections in that interval; with the exception of the cases when a reflection happens too close to $\pm T$. The error is then either $0$, or $1$ or $-1$. 
 \end{remark}
 
  \begin{remark}\label{rem_T2}
Consider the special case $\phi(t)= (T-t)/T$ for $0\le t\le T$ and $\phi=0$ otherwise. This is the function we use in our numerical experiments. 
This is not a smooth function either, but we can deal with this as above.  Then $l_k = |\tau_{k+1}-\tau_k|/T$ for $k\not=0$ if if the whole interval is in $[-T,T]$, and $l_0 = (\tau_1-\tau_{-1})/T$. If $\tau_k<T<\tau_{k+1}$, then $l_k= T-\tau_k$, similarly for $k<0$. In other words, $l_k$ are just the lengths of the geodesic segments up to time $|t|\le T$ with the first and the last ones having endpoints not on $\bo$ generically. Then Lemma~\ref{lemma_p} holds with those values, away from the rays for which the broken geodesics hits $\bo$ for $t=\pm T$. The right-hand side of Figure~\ref{fig:thermo} illustrates that if we assume that the  plane $t=0$  intersects the longest ray there. 
\end{remark}

\begin{remark} \label{rem3}
One intuitive way to explain the lemma is to look  at the representation \r{I10} of the error term $\mathcal{K}_0$. The integrand admits the following interpretation. Each singularity propagates without sign change at the time of reflection (represented by the group $\mathbf{U}_N(\tau)$. Then ignoring for a moment the projection $\mathbf{\Pi}$ (which is identity up to a smoothing operator in the interior of $\Omega$) the same singularity travels back but satisfies Dirichlet boundary conditions, thus the sign of the amplitude changes at each reflection. The result at time $t=0$ is $(-1)^k \Id$ modulo lower order terms, where $k=k(\tau)$ is the number of reflections over the interval $t\in[0,\tau]$. The integral averages those values with weight $\chi$, which explains \r{kappa}.  The difficulty in following this approach is that we have to isolate the times of reflection with small intervals (then the singularity ends at $\bo$ for that time); and this those times depends on $(x,\xi)$. 
\end{remark}
\begin{figure}[t]   
  \centering
  \includegraphics[scale=0.25, keepaspectratio]{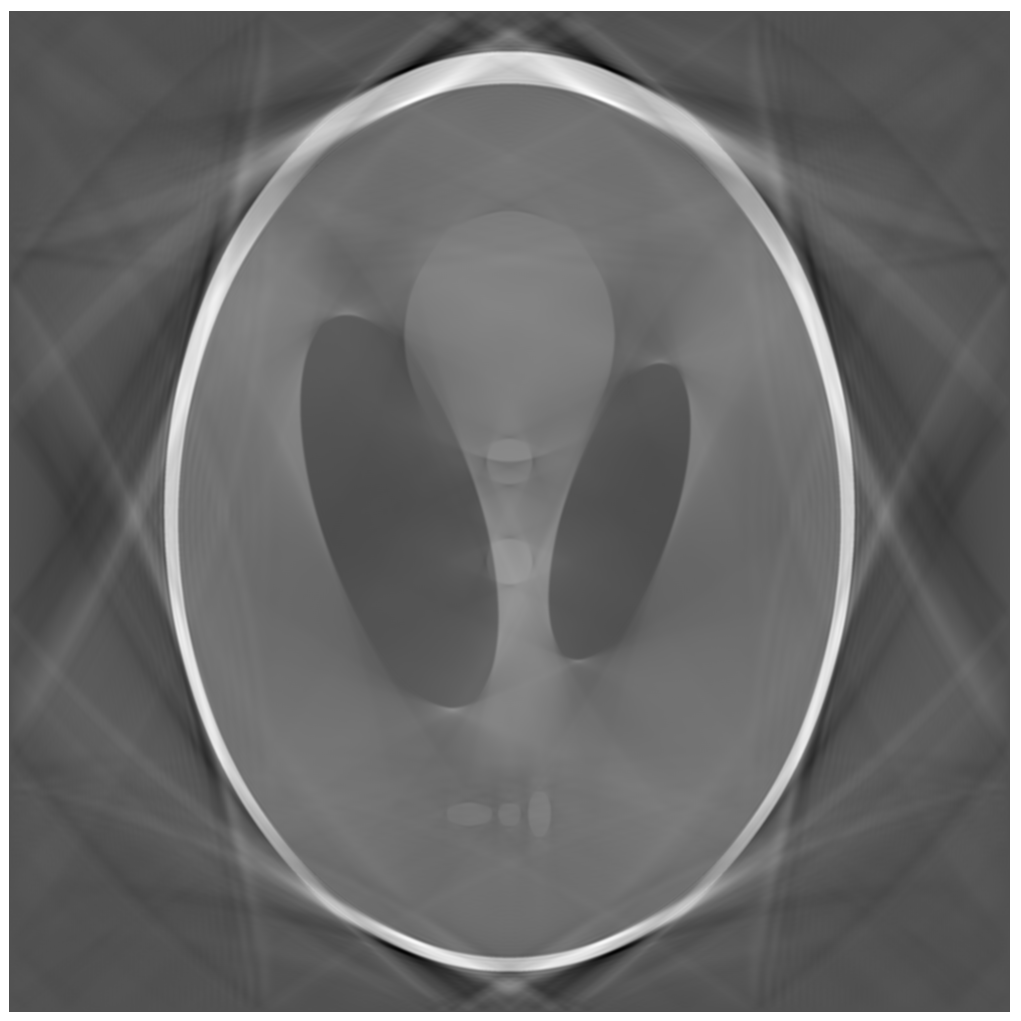}  \hspace{30 pt}
   \includegraphics[scale=0.25, keepaspectratio] {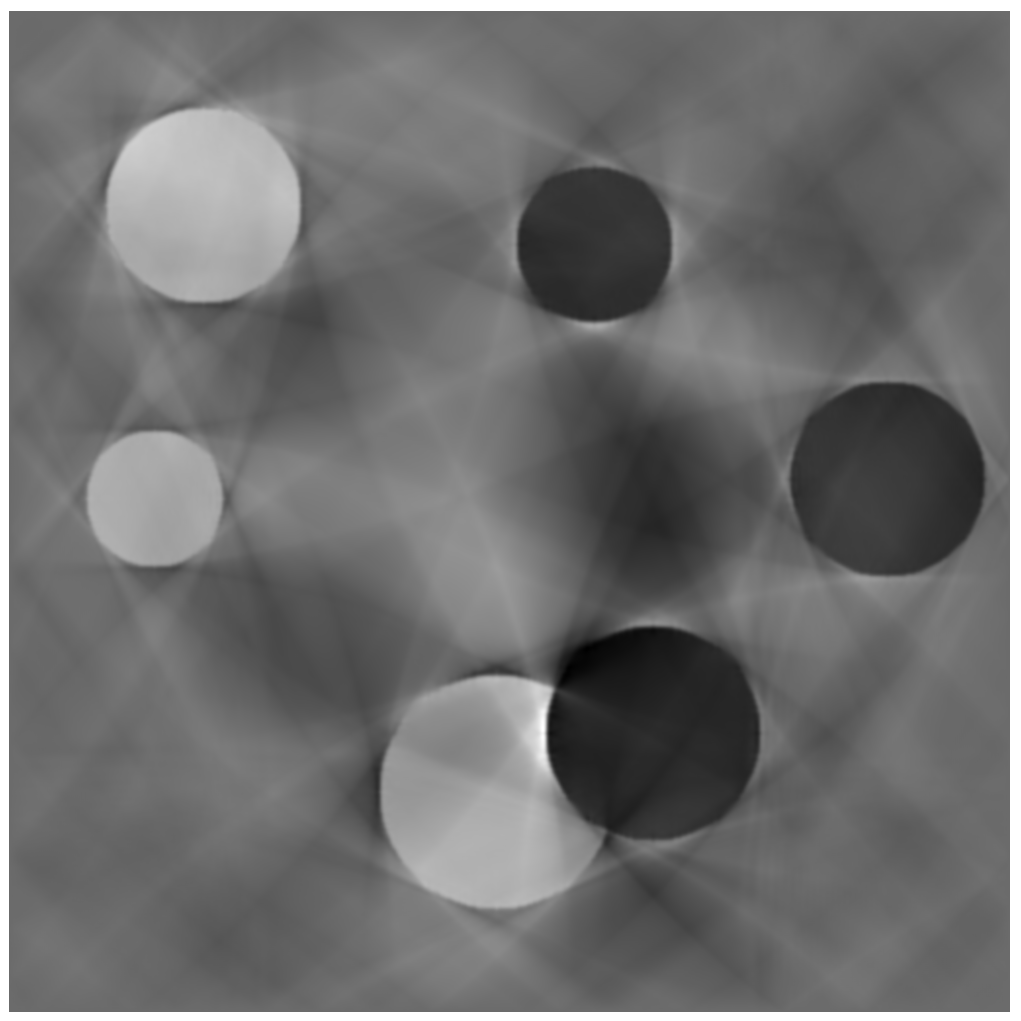}
  \caption{Sharp time reversal (with harmonic extension) and $c=1$, $T =0.9d$, where $d$ is the diagonal. The originals are the Shepp-Logan phantom and white and black disks on a gray background. 
  The purpose of this example is to illustrate the failure of the standard time reversal (with or without the harmonic extension) to resolve all singularities, see Remark~\ref{rem_T1}. Some singularities are lost (amplitude $0$), some are recovered with the right amplitude $1$, and some with amplitude $2$. The numerical range of the reconstructed images is roughly speaking twice that of the originals. 
 The harmonic extension creates a weak singularity not visible on the plots.} \label{fig1}
\end{figure}

 The following lemma is  what remained to complete the proof of Theorem~\ref{thm2.1}.

\begin{lemma}\label{lemma_R} Under the conditions of Theorem~\ref{thm2.1},   the operator $\mathcal{K}_0^* \mathcal{K}_0$ in ${H}_D(\Omega_0)$  
has an upper bound of its essential spectrum less than one. More precisely, that bound is the maximum $\kappa_0$ of $|\kappa|$ on $S^*\bar\Omega_0$. 
\end{lemma}

\begin{proof} 
Our starting point is the representation \r{h2} for the boundary values of the solution $v$ of \r{T4}. Write $h$ in the form
\be{hhh}
h(t) = \phi(t)\Lambda f(t)- \int_t^T\chi(\tau) \Lambda f(\tau)\,\d\tau. 
\ee
The second term on the right is composition of a multiplication by the smooth function $\chi$ and a particular choice of an anti-derivative w.r.t.\ $t$. The operator $\d/\d t$ is not elliptic but it is elliptic in the hyperbolic region. Then so is any of its left inverses; with a principal symbol $-\i/\tau$ restricted to the hyperbolic region away from the origin $(\tau,\xi')=(0,0)$. Then by Egorov's theorem, backprojecting that term contributes a lower order \PDO\ at $t=0$. The principal symbol contribution comes from the first term on the right then. For that, we can apply Lemma~\ref{lemma_main} to get that $f\mapsto v(0)$ in \r{AA} is a \PDO\ with the properties described in Lemma~\ref{lemma_main}. 

We have  $\A_0 = \Pi_0v(0)$ and $\mathcal{K}_0  =\Pi_0Q$, where $Q:= \kappa (x,D)$ is a properly supported \PDO\ in $\Omega$.

\begin{figure}[t] 
  \centering
  \includegraphics[scale=0.25, keepaspectratio]{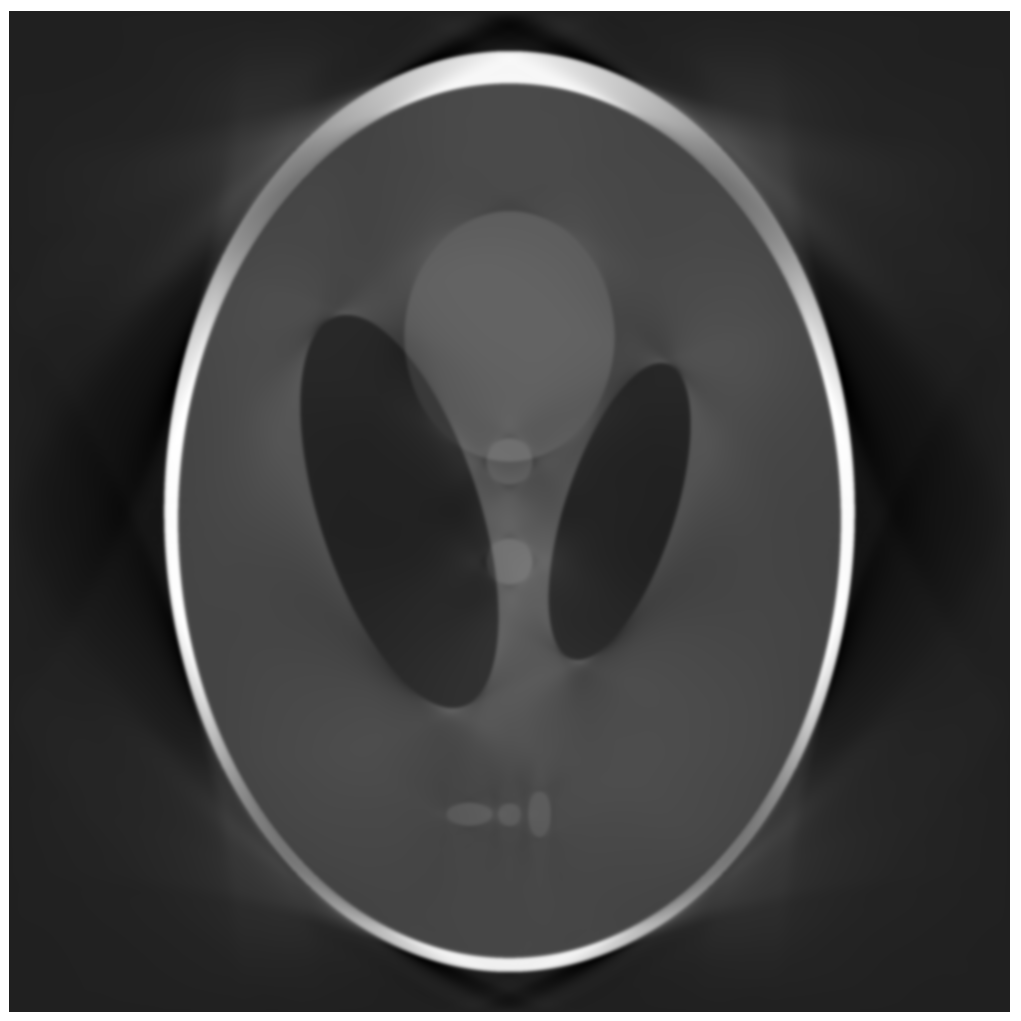}  \hspace{30 pt}
   \includegraphics[scale=0.25, keepaspectratio]{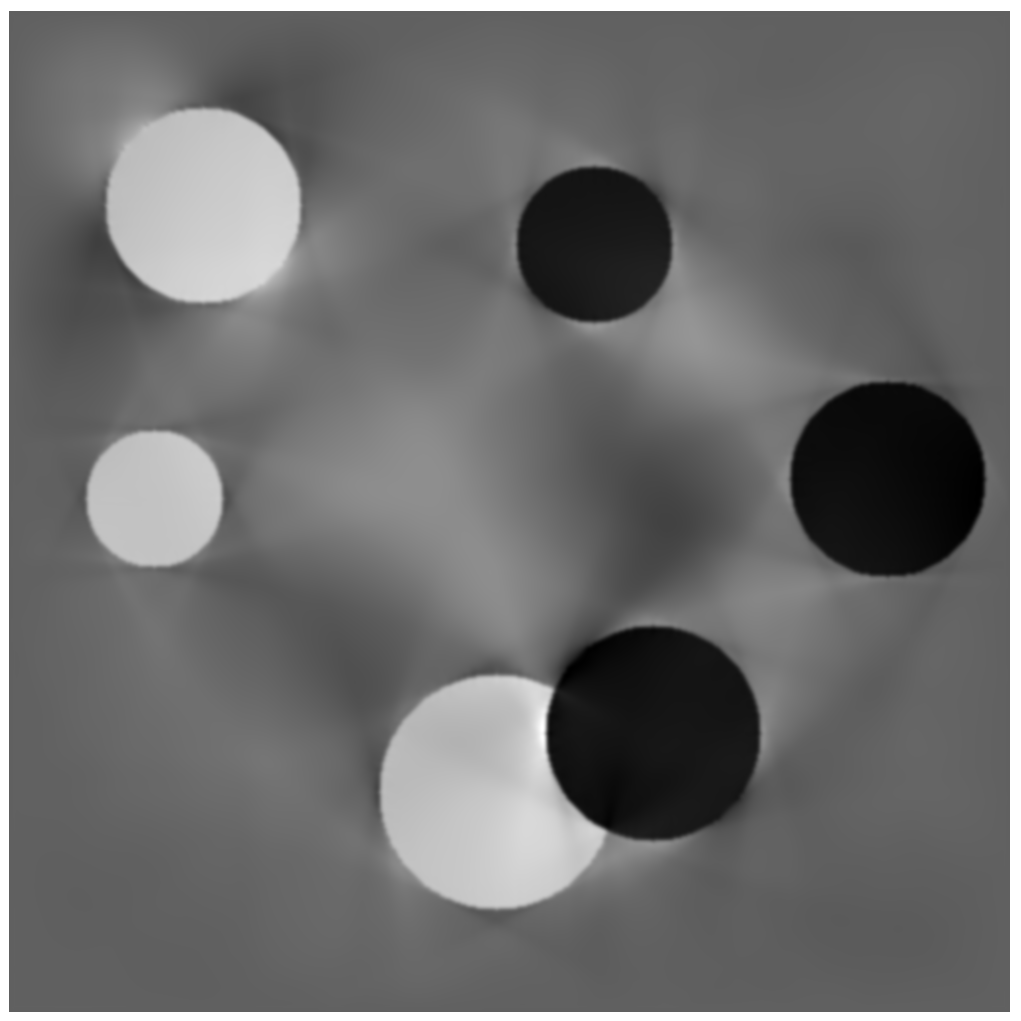}
  \caption{Averaged time reversal with a linear decreasing function form $0$ to $T =0.9d$, see \r{hhh} and Remark~\ref{rem_T2}, where $d$ is the diagonal, $c=1$. The originals are as above. 
 In this example we  illustrate the power of the averaging (without the harmonic extension  which affects the lower frequencies only). 
All singularities are reconstructed now but with different (positive) amplitudes except for those who hit a corner because the boundary is not smooth.   The lowest amplitudes are at rays close to those hitting a corner at $\pi/4$ angles as in the white disk in the upper left corner.  }   \label{fig2}
\end{figure}

In $L^2$, $Q^*Q$ is a \PDO\ with principal symbol $\kappa^2$, where $Q^*$ is the $L^2$ adjoint. To compute the same in $H_D$,  
for ${f}\in   C_0^\infty(\Omega_0)$ write
\[
  \|\mathcal{K}_0{f}\|^2_{H_D(\Omega_0)}  
  \le (P Qf,Qf)_{L^2(\Omega)}  = (Q^*PQf,f)_{L^2(\Omega)}.
\]
The principal symbol of $Q^*PQ$ is $\kappa^2\sigma_p(P)$; therefore, there is a \PDO\ $K_1$ of order $0$ so that
\[
(Q^*PQf,f)_{L^2(\Omega)} = ((Q^*Q + K_1)P^{1/2}f,P^{1/2}f)_{L^2}  .  
\] 
Since $\kappa^2\le\kappa_0^2<1$ on $S^*\bar\Omega_0$, for every $\eps>0$, 
  we have $(\kappa_0+\eps)^2-q_0^2>C>0$ in some neighborhood of $S^*\bar\Omega_0$ in $S^*\Omega$. Then $(\kappa_0+\eps)^2-\Re\text{Op}(q_0^2) = B^*B+K_2$ with $K_2$ compact, where $B$ is of order $0$,  where $\Re L = (L+L^*)/2$, see \cite[Lemma~II.6.2]{Taylor-book0}. Therefore, $Q^*Q = (\kappa_0+\eps)^2-B^*B+K_3$, with $K_3$ compact in $L^2$. In particular, $Q^*Q:L^2(\Omega_0)\to L^2(\Omega_0)  $ is a contraction up to a compact operator. 
Therefore,
\be{KK}
  \|\mathcal{K}_0{f}\|^2_{H_D(\Omega_0)} \le (\kappa_0+\eps)^2 \|{f}\|^2+\|K_4f\|\|f\| 
\ee
with $K_4$ a compact operator in $H_D(\Omega_0)$.

Assume now that the essential spectrum of $\mathcal{K}_0^* \mathcal{K}_0$ contains $\kappa_0+2\eps$. Then there is a orthonormal sequence $f_n$ so that $\mathcal{K}_0^* \mathcal{K}_0f_n = (\kappa_0+2\eps)f_n+o(1)$, see \cite[VII.12]{Reed-Simon1}. Since $K_4$ is compact, $K_4f_n\to0$. Taking the limit $n\to\infty$ in \r{KK}, we get a contradiction.
\end{proof}

\section{Numerical simulations for  data on the whole boundary} \label{sec7}
We used 1001x1001 grids and a second order finite difference scheme for Figure~\ref{fig1} and Figure~\ref{fig2}, and on a 501x501 grid for Figure~\ref{fig3}. The purpose of those tests was to illustrate the mathematics. Numerical tests under conditions that would more closely resemble the actual applications will be presented in a forthcoming work. 

The first phantom is the Shepp-Logan one, properly resampled from a higher resolution to prevent jagged edges. The second phantom are white and black disks on a uniform gray background. The iterations for Figures~\ref{fig3}, \ref{fig4} are done in the following way. 
\[
\begin{split}
f_1 &= \A_0 h, \quad h:= \Lambda f,\\
f_{n} &= (\Id-\A_0\Lambda) f_{n-1} +\A_0 h, \quad n=2,3,\dots. 
\end{split}
\]
At each step, we evaluate $\|f_{n}-f\|_{L^2}$ but we do not use this to decide how many steps to take (since $f$ is unknown). We use  $10$ terms in \r{2.2} for Figures~\ref{fig3}, \ref{fig4}.

\begin{figure}[h]   
  \centering
  \includegraphics[scale=0.5, keepaspectratio]{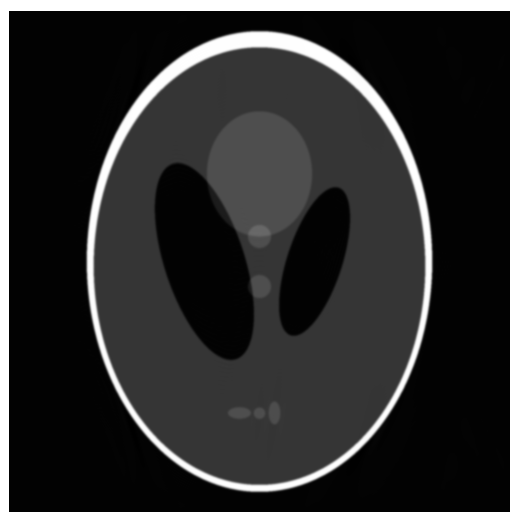}  \hspace{26 pt}
   \includegraphics[scale=0.5, keepaspectratio]{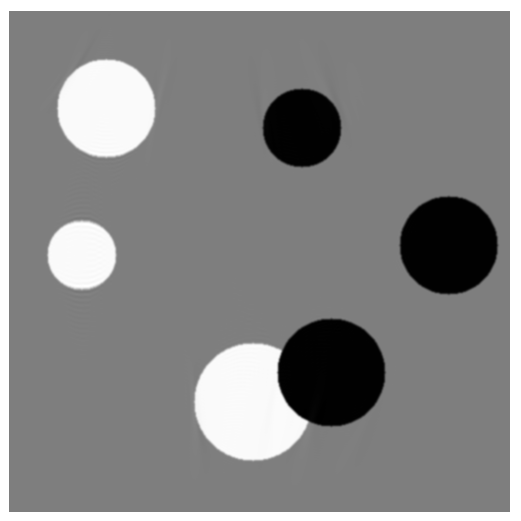} 
  \caption{Full data Neumann series inversion with $10$ terms,  $T =5$, on the square $[-1,1]^2$, variable $c= 1+0.3\sin(\pi x^1)+0.2\cos(\pi x^2)$.  The originals are as above. 
We average with $\phi$ as in Remark~\ref{rem_T2}, i.e., with weight $\chi =1$ in $[0,T]$, see \r{A}. 
The   artifacts, best visible when plotting the difference, are mainly due to the presence of corners.   The $L^2$ error on the left is $0.44\%$; and on the right: $0.34\%$. The $L^\infty$ error on the left is about $1.2\%$; and about $3\%$  on the right. }\label{fig3}
\end{figure}

The boundary here is not smooth, and singularities hitting $\bo$ too close to a corner have very short paths before the next reflection. This creates some mild instability as the spectral bound of the error there is too close to $1$. The rays hitting a corner close to 45 degrees are the most unstable. The faint artifacts in the second reconstruction in Figure~\ref{fig3} can be explained by that. 

We do not compute numerically the lower bound for the  time $T$ needed for stability in Figure~\ref{fig3}, where $c$ is variable. When $c=1$, this lower bound is  half of the diagonal, i.e., $\sqrt 2$; and then the time $T=5$ exceeds it by a comfortable margin to be able to claim that $T=5$ is enough for stability even for that choice of $c$.  Numerical experiments with $T=3$ show a very good reconstruction, as well, even with partial data as in the next section. 

%
%

\section{Partial Data}
\subsection{Sharp averaged time reversal}

Assume that we are given $\Lambda f$ on $\Gamma\subset\Omega$. We do time reversal with partial data as follows. Solve
\begin{equation}   \label{P1}
\left\{
\begin{array}{rcll}
(\partial_t^2 +P)v &=&0 &  \mbox{in $(0,T)\times \Omega$},\\
  v|_{(0,T)\times\Gamma}&=&h,\\
   \partial_\nu  v|_{(0,T)\times(\bo\setminus\Gamma)}&=&0,\\
v|_{t=T} &=&\phi,\\ \quad \partial_t v|_{t=T}& =&0, 
\end{array}
\right.               
\end{equation} 
where, eventually, we will set $h=\Lambda f$, and we will choose $\phi$ below; and set 
\be{P3}
Ah : = v(0). 
\ee
The choice of the boundary conditions is dictated by the following: we know $u$ on $(0,T)\times \Gamma$, and we use this information. Next, we do not know $u$ on the rest of $\bo$ but we know that it satisfies homogeneous Neumann boundary conditions there. To choose $\phi$, we use the same arguments: we solve the following Zaremba problem
\be{P5}
\Delta_g \phi=0, \quad \phi|_{\Gamma}=h, \quad \partial_\nu  v|_{\bo\setminus\Gamma}=0.
\ee
This is a well posed problem if the boundary data is in $L^2$ at least, see, e.g., \cite{Shamir,Schulze-book}. The Laplacian with homogeneous mixed conditions has a natural self-adjoint realization, and by the Stone's theorem, \r{P1} is well posed and energy preserving, as well. 

More precisely, let 
\[
H_Z(\Omega)= \left\{f\in H^1(\Omega);\;  f|_{\Gamma}=0 \right\},
\]
equipped with the Dirichlet  norm \r{2.0H}, 
and set $\mathcal{H}_Z(\Omega) = H_Z(\Omega)\oplus L^2(\Omega)$. On $H_Z(\Omega)$, we define the self-adjoint operator $P_Z$ with domain
\[
\mathcal{D}(P_Z) = \left\{f\in H_Z(\Omega);\;  \partial_\nu f|_{\bo\setminus \Gamma}=0 \right\}. 
\]
Then we define $\mathbf{P}_Z$ as in \r{s1} with $P=P_Z$ with domain consisting of all $\mathbf{f}\in\mathcal{H}_Z$ so that $\mathbf{P}_Z\mathbf{f}$ (considered in distribution sense) belongs to $\mathcal{H}_Z$, see also \cite{Corn_Robb_11}. Let $\mathbf{U}_Z(t)=\exp(t\mathbf{P}_Z)$ be the corresponding unitary group.  

Define the ``error'' operator $K$ as before by  
\[
A\Lambda = \Id-K. 
\]

To analyze $K$, let $w=u-v$ be the ``error''. Then $w$ solves
\begin{equation}   \label{I4b}
\left\{
\begin{array}{rcll}
(\partial_t^2 +P)w &=&0 &  \mbox{in $(0,T)\times \Omega$},\\
  w|_{(0,T)\times\Gamma}&=&0,\\
    \partial_\nu w|_{(0,T)\times(\R^n\setminus \Gamma)}&=&0,\\
w|_{t=T} &=&u|_{t=T} -\phi\\ \quad \partial_t w|_{t=T}& =&0. 
\end{array}
\right.               
\end{equation} 
Then
\be{I5a}
Kf  = w(0). 
\ee
For $f\in C^\infty(\bar\Omega)$, let  $\Pi f := f-\phi$, where $\phi$ solves \r{P5} with $h=f|_{\Gamma}$. 
Since $\Pi f$ vanishes on $\Gamma$ and $\partial_\nu\phi$ vanishes on $\bo\setminus\Gamma$, after integration by parts, we get $\Pi f\perp \phi$ in $H_Z(\Omega)$. Therefore, we have the Pythagorean identity
\[
\|\Pi f\|^2+\|\phi\|^2 = \|f\|^2
\]
in the $H_Z$ norms. In particular, $\|\Pi \|\le 1$.

This construction  yields the following for the operator $K : H_N \to   H_D$, see also \r{I6}:
\be{I6'}
K = \pi_1 \mathbf{U}_Z(-T)\mathbf{\Pi} \mathbf{U}_N(T) \pi_1^* ,
\ee
where $\pi_1(f_1,f_2) := f_1$, $\pi_1^*f :=(f,0)$. Obviously,
\be{I7b}
\|K\|_{H_N\to H_Z}\le 1. 
\ee
We define the averaged time reversal map $\A$ as in \r{A}. The latter can be also described as follows. We solve
\begin{equation}   \label{I4''}
\left\{
\begin{array}{rcll}
(\partial_t^2 +P)v &=&0 &  \mbox{in $(0,T)\times \Omega$},\\
  v|_{(0,T)\times\Gamma}&=&h(t),\\
    \partial_\nu v|_{(0,T)\times(\R^n\setminus \Gamma)}&=&0,\\
v|_{t=T} &=&0 \\ \quad \partial_t v|_{t=T}& =&0. 
\end{array}
\right.               
\end{equation} 
Here $h$ is as in \r{TR}, see also \r{h2}. Then we set $\A=v(0)$ and  $\A_0= \Pi_0 \A$. 

The analysis of $K$ in this case however is more complicated. To prove the equivalent of Lemma~\ref{lemma_R} in this case, we need to study the propagation of singularities for the Zaremba problem: those who hit the boundary of $\Gamma$. 
The convergence of the Neumann series in this case is an open problem. A numerical reconstruction is shown in Figure~\ref{fig4}. The data used is on the left and the bottom sides on the squares, and on 20\% of the other two sides, as marked there. When $c=1$, as on the left, this is a stable configuration, ignoring the fact that $\bo$ is not smooth at corners, and the critical time for stability is the diagonal $T=2\sqrt2\approx 2.82$. A critical case would be to use two adjacent sides. 
 We choose $T=5$ in both cases. 
Reconstructions with $T=3$, not shown here, look still very good, with a slightly higher $L^2$ error: 4.25\% for the Shepp-Logan phantom (on the right), vs.\ 2\% for $T=5$ and $1.5\%$ for $T=7$. 

\begin{figure}[h]   
  \centering
  \includegraphics[scale=0.5, keepaspectratio]{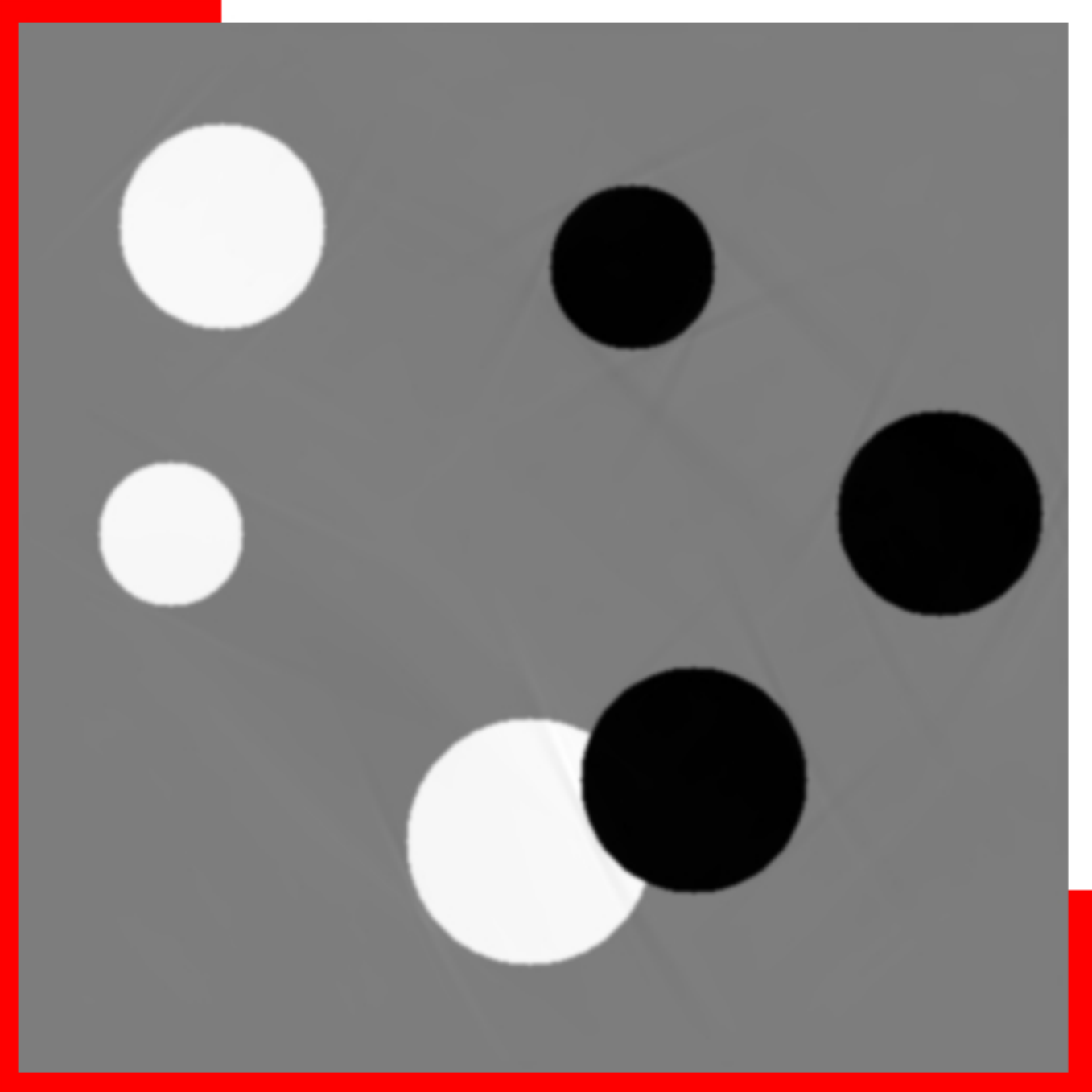}  \hspace{26 pt}
   \includegraphics[scale=0.5, keepaspectratio]{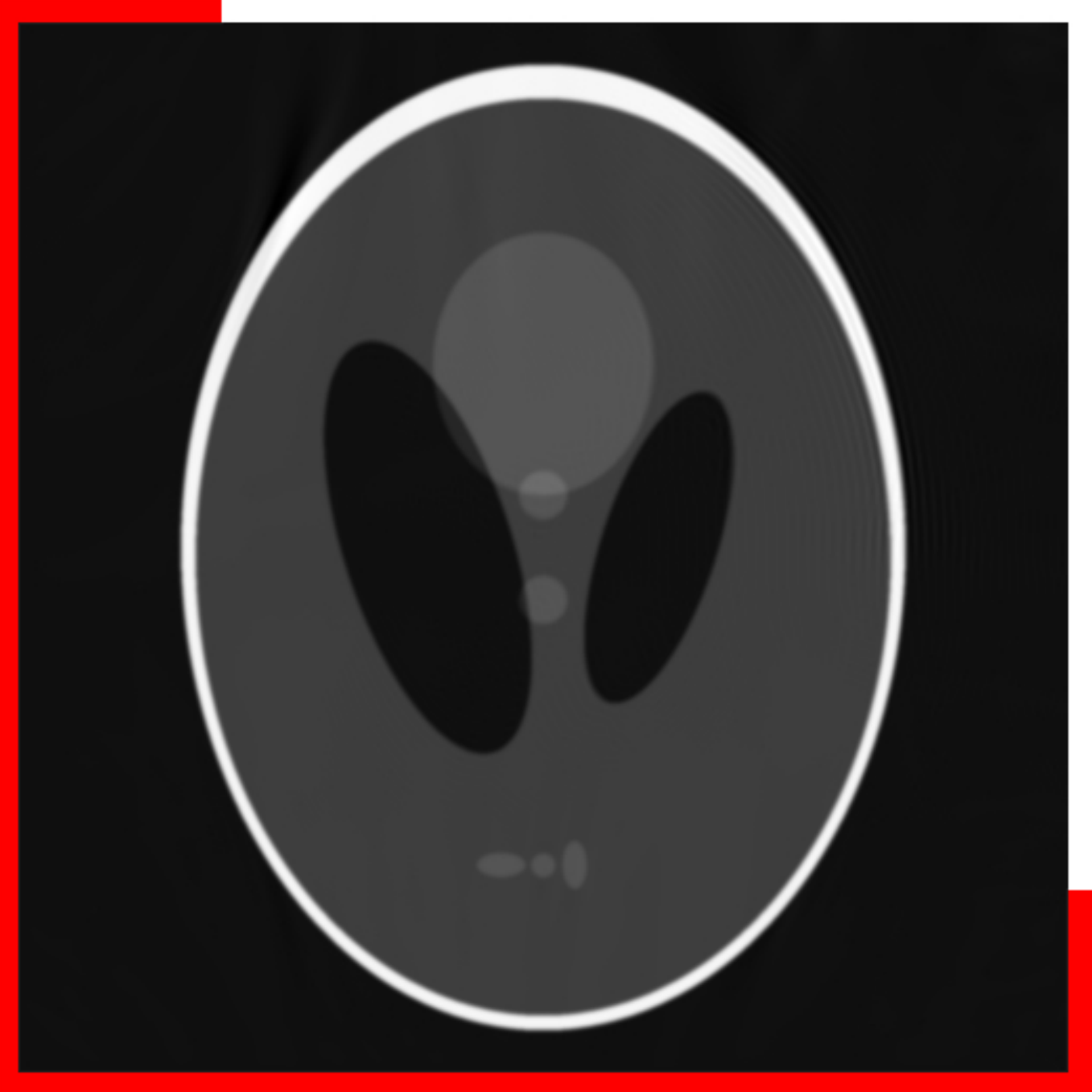} 
  \caption{Partial data inversion with data on the indicated part of $\bo$. 
  Neumann series inversion with $10$ terms,  $T =5$, on the square $[-1,1]^2$. 
  Left: constant speed $c=1$, $L^2$ error $=0.7\%$. Right:  variable $c= 1+0.3\sin(\pi x^1)+0.2\cos(\pi x^2)$, $L^2$ error $=2\%$. 
We average with $\phi$ as in Remark~\ref{rem_T2}, i.e., with weight $\chi =1$ in $[0,T]$, see \r{A}. 
Again, the  the most visible artifacts can be explained by the presence of corners.  
}
\label{fig4}
\end{figure}

\subsection{Recovery of singularities}

Instead of that, we will show that our method gives a parametrix recovering almost all singularities under the technical assumption that $f$ has no singularities hitting the edge of $\Gamma$.

For a fixed $T>0$ and $\Gamma$, let $\mathcal{V}\subset T^*\Omega\setminus 0$ be the open set of visible singularities, see Definition~\ref{def1}. Let $\mathcal{I}\subset T^*\Omega\setminus 0$ be the open set of invisible singularities. Recall that   $\Sigma_0:= T^*\bar\Omega_0\setminus (\mathcal{V}\cup\mathcal{I})$ is a conic set  of measure zero.

Then the proof of Lemma~\ref{lemma_R} implies the following.

\begin{proposition}\label{pr1} 
Let $\mathcal{U}\Subset  \mathcal{V}\cup\mathcal{I}$ be an open conic set. 
Let $f$ be supported in $\bar\Omega_0$ and let $\WF(f)\cap \Sigma_0=\emptyset$. 
Then there exist a \PDO\ $M$ of order $0$ with a homogeneous principal symbol taking values in $[1-\kappa_0,1+\kappa_0]$,    $\kappa_0\in [0,1)$ in $T^*\mathcal{U}$, and essential support in $\mathcal{V}$, so that
\[
 \A\Lambda f = Mf\quad \text{mod $C^\infty$}. 
\]

In particular, if the stability condition holds, $\mathcal{I}=\emptyset$, and $M$ is elliptic away from $\Sigma_0$. 
\end{proposition}

\begin{proof}
We follow the proof of Lemma~\ref{lemma_R}. The unit speed geodesic issued from each $(x,\xi)\in\mathcal{V}$ hits $\bo$ at a point either on $\Gamma$  or on $\bo\setminus\bar \Gamma$, for $|t|\le T$. When backprojecting the Dirichlet data, the back-propagating geodesics hits $\bo$ at the same points. Let is rename the reflection times $\tau_k$ by calling $\tau_1$ the first time for which the geodesic hits $\Gamma$ (ignoring those where it hits $\bo\setminus\bar \Gamma$), etc. Then at the reflection times related to $\bo\setminus\bar \Gamma$, the principal part does not change sign because we imposed Neumann boundary conditions there. At the remaining ones, it does. Therefore, all the arguments hold and Lemma~\ref{lemma_main} and Lemma~\ref{lemma_p} still hold with the so redefined $\tau_k$. If $(x,\xi)$ is visible, then there is at least two terms in the sum in \r{kappa} which proves the proposition.
\end{proof}

The proposition implies that $\A$  recovers the visible part of $\WF(f)$  under the a priori assumption that $\WF(f)$ is disjoint from $\Sigma_0$. Also, $\kappa_0$ can be chosen as in Lemma~\ref{lemma_R} with $\tau_k$ as in the proof above. Next, writing $M=\Id-K$, the formal Neumann expansion $\Id+K+K^2+\dots$ applied to $\A\Lambda f$, considered in Borel senses, recovers $f$ microlocally in $\mathcal{U}$. The invisible singularities, those in $\mathcal{I}$, cannot be recovered.  In practical reconstructions, a finite expansion with $N$ terms recovers $f$ microlocally there approximately with an exponential error of the principal symbol. 

Finally, we notice that general microlocal arguments like those used in \cite{BardosLR_control}, imply that one can recover all visible singularities in a stable way.
Our goal here was to suggest a constructive way of doing so.  When the observations are done on the whole boundary, stability follows from  Theorem~\ref{thm_st} but in Theorem~\ref{thm2.1}, we show how to reconstruct $f$ in a stable way. 



\end{document}